\documentclass[11pt]{article} 
\usepackage{graphicx} 

\usepackage{amssymb}
\usepackage{amsmath}
\usepackage{setspace}
\usepackage{xcolor}

\usepackage{subcaption}

\usepackage{hyperref}

\newcommand{\comm}[1]{}

\usepackage[fleqn,tbtags]{mathtools}

\usepackage{amsthm}
\newtheorem{thm}{Theorem}[section]

\newtheorem{cor}[thm]{Corollary}
\newtheorem{lemma}[thm]{Lemma}
\newtheorem{prop}[thm]{Proposition}

\newtheorem{defn}[thm]{Definition}
\newtheorem{remark}[thm]{Remark}
\newtheorem{example}[thm]{Example}

\numberwithin{equation}{section}
\counterwithin{figure}{section}

\usepackage{tikz}
\usepackage{pgfplots}

\usepackage{enumitem}

\def\Xint#1{\mathchoice
{\XXint\displaystyle\textstyle{#1}}%
{\XXint\textstyle\scriptstyle{#1}}%
{\XXint\scriptstyle\scriptscriptstyle{#1}}%
{\XXint\scriptscriptstyle\scriptscriptstyle{#1}}%
   \!\int}
\def\XXint#1#2#3{{\setbox0=\hbox{$#1{#2#3}{\int}$}
\vcenter{\hbox{$#2#3$}}\kern-.5\wd0}}

\def\dashint{\Xint-}

\DeclareMathOperator{\spt}{spt}
\DeclareMathOperator{\proj}{proj}

\DeclareMathOperator{\Id}{Id}
\DeclareMathOperator{\conv}{conv}

\title{Equilibria of the Pressureless Euler System in Dimension $1$}
\author{Nicholas Biglin, Columbia University, nb3101@columbia.edu \\ Joseph Crachiola, Wayne State University, ho9872@wayne.edu\\ Thomas Kunz, University of Minnesota, kunz0119@umn.edu \\Omkar Maralappanavar, University of Connecticut, omkar.maralappanavar@uconn.edu}
\date{}

\begin{document}

\maketitle

\singlespacing
\begin{abstract}
    The pressureless Euler equations in dimension $1$ describe the evolution of a distribution of mass over the real line. We are interested in sticky particles solutions: distributional solutions in which mass that collides sticks together for all time. Given initial data, a projection formula is known to give a sticky particles solution. We obtain a necessary and sufficient condition for this solution to go to equilibrium based only on the initial condition. Furthermore, we give a precise characterization of the equilibrium, as well as a formula for the time of collapse.  
\end{abstract}

\doublespacing

\section{Introduction and Overview}\label{Overview}
    The pressureless Euler system was originally proposed by Zeldovich \cite{Zeldovich} to model the formation of large structures in the universe. In one spatial dimension, it is given by 
    \begin{align}
        \begin{cases}\label{PE}\tag{PE}
        \partial_t \rho + \partial_y (\rho v) = 0 
        \\
        \partial_t (\rho v) + \partial_y (\rho v^2) = 0,
        \end{cases}
    \end{align}
    where $\rho$ represents a distribution of mass on the line and $v$ gives the velocity of the mass at any point. The first equation, known as the continuity equation, prescribes local conservation of mass, while the second equation prescribes conservation of momentum. $\rho_t$ (denoting $\rho$ at time $t$) is to be understood as a probability measure on $\mathbb{R}$ of finite second moment.\footnote{While we do not consider it here, it is also natural to study \eqref{PE} with a sticky boundary condition, see \cite{Tudorascu2023}.}
    Thus, the equations and initial conditions are to be understood in the sense of distributions. We investigate Sticky Particle Solutions (SPS), solutions which imposes that mass sticks together at collisions.\footnote{The sticky condition can be stated precisely as follows: for all $t\geq s\geq0$, $N_s(a)=N_s(b)$ implies that $N_t(a)=N_t(b)$, where $(N_t)_\#\mathcal{L}=\rho_t$. (The optimal map $N_t$ is defined in section \ref{OptimalTransport-section}.)}

    The analysis of the pressureless Euler system was first undertaken in \cite{ERS}, \cite{BrenGren}. Since then an important development is the projection formula in \cite{Natile-Savare} (a key tool in this paper, see Section \ref{projection-section}). A discussion is in order regarding well-posedness. The question of existence was settled in \cite{Tudorascu-2008}, \cite{Hynd2019}, and \cite{hynd2020trajectorymappressurelesseuler}. \cite{Suder-Tudorascu} is (as of present) the most general uniqueness result. Specifically, SPS and SCL solutions discussed in \cite{BrenGren}, \cite{Tudorascu-2008}, are linked to a result in \cite{Huang-Wang}, obtaining that SPS are unique in the case that $v_0$ is $L^\infty(\rho_0)$ and right-continuous. This involved strengthening the (a.e.) Oleinik condition (well-established in the literature discussed here) to an \textit{everywhere} Oleinik condition. \cite{Natile-Savare} also gives results on the stability of solutions. This paper's most immediate relation is \cite{Hynd-Tudorascu}, which shows that an asymptotic limit $\rho_\infty$ exists in the case that $\rho$ remains within a compact set for all time. Such an assumption is not necessary for these results. 
    
    The right side of the momentum equation may be replaced by a potential, yielding, most notably, the attractive and repulsive Euler-Poisson equations. In either, the only equilibrium is a single point mass (which is always achieved for compactly supported solutions in the attractive case as can be seen in \cite{BLPTW}). Equilibrium in the case of \eqref{PE}, however, appears more rich, as any distribution on $\mathbb{R}$ can be an equilibrium solution. This raises three questions about sticky particles solutions to the preasureless Euler problem, which we will answer in this paper:  
    \begin{enumerate}
        \item Under what conditions does a SPS go to an equilibrium? 
        \item In the case that a SPS goes to equilibrium, what distribution does it approach?  
        \item In the case that a SPS goes to equilibrium, what is the time of collapse? 
    \end{enumerate}
    All three of these questions can be answered completely using only the initial conditions $(\rho_0,v_0)$. Of course, further exploration of the other cases of the Pressureless Euler-Poisson equations is still in order. As for the one dimensional case, the attractive equations can be studied using similar techniques to this paper, (such as the projection formula) as can be seen in \cite{BLPTW}. However, the repulsive equations exhibit more pathological behavior, resisting such methods (see our recent paper, \cite{Repulsive-Paper}). Thus, more analysis is needed in that case; in particular, we still lack a method to predict whether a solution will go to equilibrium based on the initial conditions.

\section{Preliminaries}
\subsection{The Wasserstein Distance and Optimal Transport}\label{OptimalTransport-section}
    Throughout this paper, we denote the Lebesgue measure on $(0,1)$ by $\mathcal{L}$. We denote the class of probability measures with finite $p$-moment as follows:
    \begin{align*}
        \mathcal{P}_p(\mathbb{R})\coloneq\big\{\mu \text{ a Borel probability on }\mathbb{R}\ \big|\ \int y^p\mu(dy)<\infty\big\}.
    \end{align*}
    As seen in \cite{Villani}, for a Borel map $S:\mathbb{R}\to\mathbb{R}$ and probability $\mu$ on $\mathbb{R}$, the \textbf{push-forward} of $\mu$ by $S$, denoted $S_\#\mu$, is the (unique) Borel probability such that for all $\varphi\in C_c(\mathbb{R})$,
    \begin{align*}
        \int\varphi(y)S_\#\mu(dy)=\int\varphi(S(x))\mu(dx). 
    \end{align*}
    To precisely describe convergence to equilibrium, we will need the Wasserstein metric, which describes a distance between probabilities on $\mathbb{R}$. The precise definition of the \textbf{$p$-Wasserstein distance} between measures $\mu_1,\mu_2\in\mathcal{P}_p(\mathbb{R})$ may be found in \cite{Villani}. 
    Here, we will compute the $p$-Wasserstein distance using this important fact: if $\mu_1$ and $\mu_2$ have \textbf{optimal maps} $N_1, N_2\in L^p(0,1)$, (the unique, right-continuous functions on $(0,1)$ with ${N_i}_\#\mathcal{L}=\mu_i$) then  
    \begin{align*}
        W_p(\mu_1, \mu_2)=\Big(\int_0^1 \big|N_1(x)-N_2(x)\big|^p \, dx\Big)^{1/p}=\|N_1-N_2\|_{L^p(0,1)}.
    \end{align*}
\subsection{The Projection Formula}\label{projection-section}
    In what follows, $\proj_{\mathcal K}:L^2(0,1)\to \mathcal{K}$ is the projection onto $\mathcal{K}$, the convex cone of nondecreasing functions on $(0,1)$. Given $f\in L^2$, it is calculated (as seen in Theorem 3.1, \cite{Natile-Savare}) by 
    \begin{align}\label{projection-computation}
        \proj_\mathcal{K}f=\frac{d^+}{dx}F^{**}(x)=\frac{d^+}{dx}\conv\Big(\int_0^x f(w) \, dw\Big)
    \end{align}
    where $F^{**}\coloneq\conv(F)$ is the convex envelope (the greatest convex function bounded above) of $F(x):=\int_0^x f(w)dw$, the primitive of $f$, and $\frac{d^+}{dx}$ is the right derivative. The following are among several useful properties, which are in fact true of any projection onto a convex cone $\mathcal{C}$ in a Hilbert space $H$: (see Lemma 3 in \cite{Ingram-Marsh}) for $f,g\in H$, 
    \begin{enumerate}[label=\textup{(\roman*)}]
        \item \hypertarget{PositiveHomogeneity}{\textup{(Positive Homogeneity)}} if $\alpha\geq0$, $\proj_\mathcal{C}(\alpha f)=\alpha\proj_\mathcal{C}f$;
        \item \hypertarget{Contraction}{\textup{(Contraction)}} $\|\proj_\mathcal{C}f-\proj_\mathcal{C}g\|_H\leq\|f-g\|_H$.
    \end{enumerate}
    
    The following theorem, the \textbf{projection formula}, is a key result from \cite{Natile-Savare}, which we will make extensive use of. 
    \begin{thm}[Natile-Savaré]\label{ProjectionFormula}
        {Let $\rho_0\in\mathcal{P}_2(\mathbb{R})$, $v_0\in L^2(\rho_0)$. The time evolution of a SPS $(\rho, v)$ to \eqref{PE} with initial conditions $(\rho_0, v_0)$, $N_{0\#}\mathcal{L}=\rho_0$, $V_0\coloneq v_0\circ N_0$, is captured by 
        \begin{align*}
            N_t=\proj_{\mathcal{K}}(N_0+tV_0) \quad \text{and} \quad V_t=v(t,\cdot)\circ N_t \quad \text{where} \quad {N_t}_\#\mathcal{L}=\rho_t \quad\text{and}\quad V_t=\frac{d^+}{dt}N_t. 
        \end{align*}}
    \end{thm}
    The main idea of this result is that, by working with the optimal map $N_t$, we translate our problem from working in ``position-space," where the independent variable is position, to ``mass-space," where the independent variable is (cumulative) mass. 
    This has the advantage of letting us work on a unit interval (even if $\spt(\rho_t)$ is unbounded) with the Lebesgue measure. The sticky particle constraint ensures that mass stays ordered (i.e., if some amount of mass begins to the left of some other amount of mass, then the initial leftward mass will remain on the left of the other mass at all times). Succinctly, $N_t$ will always be nondecreasing, hence why we can project onto $\mathcal K$. 

    Consider initial conditions $(\rho_0,v_0)$. Importantly, Theorem \ref{ProjectionFormula} gives \textit{a} solution to \eqref{PE}. In the case that $v_0$ is $L^\infty(\rho_0)$ and right-continuous, (as discussed in Section \ref{Overview}) \cite{BrenGren}, \cite{Huang-Wang}, \cite{Tudorascu-2008}, and \cite{Suder-Tudorascu} establish uniqueness of this solution; however, no such result exists in the more general case, $v_0\in L^2(\rho_0)$. Throughout the rest of this paper, we will refer to the solution generated by the projection formula in Theorem \ref{ProjectionFormula} as \textit{the} sticky particle solution (SPS). A general result on uniqueness for SPS would enhance these results: as of now, Theorem $\ref{ProjectionFormula}$ gives only \textit{a} solution in some cases, and it is this solution to which the results of this paper apply. 
    
\subsection{Some Brief Convex Analysis}

    Analysis of the projection formula will be key to our approach. The projection formula relies on the use of the convex envelope, and as such it is necessary to recall some basic properties of convex functions of one real variable. The following results can be found in \cite{Roberts} and \cite{Rockafellar}.  
    \begin{prop}\label{convex-properties} \!\!\!\! \footnote{\textit{Proof.} \cite{Roberts} Section 11.}
        Suppose that $f:[0,1]\to\mathbb{R}$ is convex. Then, $f$ has the following properties:
        \begin{enumerate}[label=\textup{(\roman*)}]
            \item $f$ is continuous on $(0,1)$ and absolutely continuous on all $[a,b]\subseteq(0,1)$;
            \item $f$ is left- and right-differentiable everywhere on $(0,1)$;
            \item $f$ is differentiable almost everywhere;
            \item The right derivative of $f$ is right-continuous and nondecreasing. 
        \end{enumerate}
    \end{prop}
    \begin{prop}\label{pointwise-sup} \!\!\!\! \footnote{\textit{Proof.} \cite{Roberts} Theorem 13 D.}
        Suppose $\{f_i:[0, 1]\to\mathbb{R}\}$ is a collection of convex functions and let $f(x)=\sup_i f_i(x)$. 
        If $J\coloneq\{x \ | \ f(x)<\infty\}$ is nonempty, then $J$ is an interval and $f$ is convex on $J$. 
    \end{prop}

    
    \begin{thm} \!\!\!\! \footnote{\textit{Proof.} Consider the $1$-dimensional case of Theorem 24.5 in \cite{Rockafellar}. The general version of the theorem presented in the text refers to the more general notion of the \textit{subdifferential}.}
        Let $f$ be convex and finite on $(0,1)$. Let $\{f_i\}_i$ be a sequence of convex functions on $(0,1)$ converging pointwise to $f$. Let $x^*\in I$ and $\{x_i\}_i$ a sequence in $I$ converging to $x^*$. Then, for all $\varepsilon>0$ there exists an index $i_0$ so that 
        \begin{align*}
            \Big[\frac{d^- f_i}{dx}(x_i),\frac{d^+ f_i}{dx}(x_i)\Big] \subset \Big[\frac{d^- f}{dx}(x^*)-\varepsilon,\frac{d^+ f}{dx}(x^*)+\varepsilon\Big] \quad \text{ for all }i\geq i_0. 
        \end{align*}
    \end{thm}
    \begin{cor}\label{DerivSwap}
        Let $f$ be convex and finite on $(0,1)$. Let $\{f_i\}_i$ be a sequence of convex functions on $(0,1)$ converging pointwise to $f$. Let $\mathcal{M}$ be the set where $f$ is differentiable (a.e. in $(0,1)$). Then, 
        \begin{align*}
            \lim_{i\to\infty} \frac{d^+}{dx}f_i(x)=f'(x) \quad \text{on } \ \mathcal{M}
        \end{align*}    
        Further, in $(0,1)\backslash \mathcal{M}$, the right derivative of $f$ is uniquely specified by the $f_i$.
    \end{cor}

    \begin{proof}
        For each $x^*\in \mathcal{M}$, take $x_i=x^*$ for all $i$ and apply the above theorem. 
        Unique specification for $x^*\notin \mathcal{M}$ comes from applying right continuity and taking a monotone decreasing sequence $x_i\to x^*$ such that $f'(x_i)$ exists for for all $i$ and passing the limit through by right continuity. 
    \end{proof}

\subsection{Galilean Invariance of PE}
    Solutions to \eqref{PE} are Galilean invariant in both position and velocity. While we do not make use of position-invariance, we will often want to normalize solutions to have $0$ total momentum. Define the total momentum of the initial conditions $(\rho_0,v_0)$ by 
    \begin{align*}
        p(\rho_0,v_0)\coloneq\int v_0(y)\rho_0(dy). 
    \end{align*}
    Now, this property goes as follows: $(\bar{\rho}_t,\ \bar{v}_t)\coloneq\big((\text{Id}-tp)_\#\rho_t,v(t,\cdot+tp)-p\big)$ is the SPS to \eqref{PE} for the initial condition $(\rho_0,v_0-p)$ if and only if $(\rho_t,v(t,\cdot))$ is the SPS for the initial condition $(\rho_0,v_0)$. We will state and prove the result in terms of the optimal maps. 
    \begin{prop}[Galilean Momentum-Invariance]\label{Galilean-Invariance}
        Suppose $\rho_0\in\mathcal{P}_2(\mathbb{R})$ and $v_0\in L^2(\rho_0)$, and consider $(N_0,V_0)$, where ${N_0}_\#\mathcal{L}=\rho_0$, $V_0\coloneq v_0(N_0)$. Then, 
    \begin{align*}
        (\bar{N}_t,\bar{V}_t)\coloneq\big(N_t-tp(\rho_0,v_0),\bar{V}_t-p(\rho_0,v_0)\big)
    \end{align*}
    gives the SPS $(\bar{\rho},\bar{v})$ to \eqref{PE} originating at $\big(N_0,V_0-p(\rho_0,v_0)\big)$ if and only if $(N_t,V_t)$ captures the SPS originating at $(N_0,V_0)$. Furthermore, 
    \begin{align*}
        \bar{Y}(t)\coloneq\int_0^1 \bar{N}_t \, dx=\int y\rho_t(dy)
    \end{align*}
    is constant, and for all $t\geq0$,
    \begin{align}\label{0-totalMomentum}
        \int_0^1 \bar{V}_t(x) \, dx=\int \bar{v}_t(y) \bar{\rho}_t(dy)=0.
    \end{align}
    \end{prop}
    \begin{proof} \!\!\!\! \footnote{Using the projection formula, as we do here, yields a short proof. However, it is not needed in order to establish this property; see the proof of Proposition 2.2, 2 in \cite{Repulsive-Paper}, where no projection formula is available.}
        The projection formula (Theorem \ref{ProjectionFormula}) tells us that all we need to establish is that $N_t-tp=\proj_{\mathcal{K}}\big(N_0+t(V_0-p)\big)$ if and only if $N_t=\proj_{\mathcal{K}}(N_0+tV_0)$. Fix $t>0$ and let $k=-tp(\rho_0,v_0)$. Observe that since, for $G_t$ the primitive of $N_0+tV_0$, 
        \begin{align*}
            \conv\big(G_t+&k\Id\big)(x)=\sup\big\{a+bx \ \big| \ a,b\in\mathbb{R}\text{, }a+bu\leq G_t(u)+ku \ \forall u\in[0,1]\big\}
            \\
            &=\sup\big\{a+(b-k)x+kx \ \big| \ a,(b-k)\in\mathbb{R}\text{, }a+(b-k)u\leq G_t(u) \ \forall u\in[0,1]\big\}
            \\
            &=\conv\big(G_t\big)(x)+kx, 
        \end{align*}
        we have $\proj_{\mathcal{K}}(N_0+tV_0+k)=\proj_{\mathcal{K}}(N_0+tV_0)+k$, (referencing \eqref{projection-computation}) from which the first part of the result follows. Observe also that for all $t\geq0$, 
        \begin{align*}
            \bar{Y}(t)=\int_0^1 N_t(x)-tp(\rho_0,v_0) \, dx=\int_0^1 N_t(x) \, dx-tp(\rho_0,v_0)=\bar{Y}(0) 
        \end{align*}
        since, using the projection formula and Proposition \ref{convex-properties}, one may verify that for fixed $t$,
        \begin{align*}
            \int_0^1 N_t(x) \, dx&=\lim_{h\to0^+}\int_h^{1-h}\frac{d^+}{dx}\conv\Big(\int_0^x N_0(w)+tV_0(w)dw\Big) \, dx=\bar{Y}(0)+tp(\rho_0,v_0),
        \end{align*}
        because the convex envelope is absolutely continuous on $[h,1-h]
        \subset(0,1)$. This also gives us \eqref{0-totalMomentum} by the weak formulation of the continuity equation, given by 
        \begin{align*}
            \frac{d}{dt}\int\varphi(y)\rho_t(dy)=\int\varphi'(y)v(t,y)\rho_t(dy) \quad \text{for all} \quad \varphi\in C_b^1(\mathbb{R}), 
        \end{align*}
        with $\varphi(y)=y$. (In other words, $0$ total momentum is equivalent to constant center of mass.)
    \end{proof}
    Because of this property, we will consider, \textit{unless explicitly stated otherwise}, solutions of $0$ total momentum, i.e., for all $t\geq0$,
    \begin{align}\label{0-total-momentum}
        \int v(t,\cdot) \, d\rho_t=\int_0^1 V_t(x) \, dx=0.
    \end{align}
\section{Characterization of Equilibrium}
    \subsection{Our Condition and Its Motivation}



    
    Consider initial conditions $\rho_0\in\mathcal{P}_2(\mathbb{R})$ and $v_0\in L^2(\rho_0)$; equivalently, if we have the optimal map ${N_0}_\#\mathcal{L}=\rho_0$, $N_0\in L^2(0,1)$ is nondecreasing and right-continuous, and $V_0\in L^2(0,1)$, where $V_0\coloneq v_0(N_0)$. Before we begin, we state precisely how we use terminology regarding equilibrium: 
    \begin{defn}\label{equilibrium-def}
        $(\rho,v)$ \textbf{converges to equilibrium} if \mbox{$W_1(\rho_t, \rho_\infty)\to0$} as $t\to\infty$ for some \mbox{$\rho_\infty\in \mathcal{P}_2(\mathbb{R})$}, or equivalently, \mbox{$\|N_t-N_\infty\|_{L^1(0,1)}\to0$}, $N_\infty\in L^2(0,1)$.\footnote{See Section \ref{p>1} for consideration of $W_p$ for $p>1$.} If, additionally, there exists $T\geq0$ such that $\rho_t=\rho_\infty$, ($N_t=N_\infty$) for all $t\geq T$, we say that $(\rho,v)$ \textbf{collapses to equilibrium in finite time}. The minimal such $T$ is the \textbf{time of collapse}.
\end{defn}
    
    Now, define the functions $F$ and $G_t$ (for fixed $t\geq0$) by 
    \begin{alignat}{2}\label{F-def}
        F(x)\coloneq&\int_0^xV_0(w) \, dw, \quad &x\in[0,1],
        \\
        \label{G-def}
        G_t(x)\coloneq&\int_0^x N_0(w) \, dw+t\int_0^xV_0(w) \, dw, \quad &x\in[0,1].
    \end{alignat}
    (They are well defined because $\rho_0\in\mathcal{P}_2(\mathbb{R})$ and $v_0\in L^2(\rho_0)$ so $N_0,V_0\in L^2(0,1)$.) Denote $\conv(F)$, the greatest convex function bounded above by $F$, as $F^{**}$. Similarly, $G_t^{**}\coloneq \conv(G_t)$.

    Throughout this paper, we call \eqref{star} the condition that 
    \begin{align}\label{star}\tag{$\bigstar$}
        F(1)=0 \quad\text{and}\quad F(x)\geq0 \quad\text{for all}\quad x\in(0, 1). 
    \end{align}
    We aim to prove that \eqref{star} is equivalent to convergence to equilibrium of the SPS to \eqref{PE} with initial conditions given by $(N_0, V_0)$.

    The first part of the condition, $F(1)=0$, is simply restating \eqref{0-total-momentum}. One may observe that, if we take $(\bar{\rho},\bar{v})$ as in Proposition \ref{Galilean-Invariance}, the solution $(\rho,v)$ has \textbf{center of mass} $Y(t)\coloneq\bar{Y}(0)+tp(\rho_0,v_0)$. It is therefore clear that convergence to an equilibrium cannot occur if $F(1)\neq0$.\footnote{Of course, it may be the case that $W_2(\rho_t,\big(\Id+tp)_\#\rho_\infty\big)\to0$ as $t\to\infty$ for some stationary distribution $\rho_\infty\in\mathcal{P}_2(\mathbb{R})$. In this case, $\bar{\rho}_t$ goes to equilibrium, and $\rho_t$ approaches a stable (just not stationary) distribution.} Thus, we will consider solutions for which $F(1)=0$ and focus on the second part of $\eqref{star}$.

    This condition has a clear physical interpretation: the integral of velocity in mass space is precisely the momentum of the initial distribution up to that certain level of mass. What \eqref{star} prescribes is that the total momentum is $0$ and the momenta of all ``left-systems" are always nonnegative. The idea is that if at any point there were a leftmost ``chunk of mass" with negative total momentum, 
    the mass to the right could only impose more leftward momentum (via collision); thus, that ``chunk of mass" must continue leftwards for all time (and simultaneously, if the total momentum of the system is 0, there must be other mass also moving ever rightwards). \textit{Scrumptiously}\comm{{\color{red} you guys are gonna make me remove this but I think it's a good word in this context} No, it is appropriate here. - Joseph}, we will see that this momentum term appears naturally in our proof, resulting from the projection formula. Indeed, due to the lack of other forces it is sensible that the momentum would completely determine the system (such an approach fails in, for example, the repulsive pressureless Euler-Poisson system, where mass exerts an outwards force, see \cite{Repulsive-Paper}). 

    \eqref{star} is equivalent to $V_0\in\mathcal{K}^\circ\coloneq\big\{f\in L^2(0,1) \ \big| \ \langle f,g\rangle=0 \text{ for all } g\in\mathcal{K}\}$. $\mathcal{K}^\circ$ is the \textbf{polar cone} of $\mathcal{K}$ discussed in \cite{Natile-Savare}. An alternative characterization of the polar cone is the set of zeros of $\proj_{\mathcal{K}}$. Thus, our condition can be equivalently written as $\proj_{\mathcal{K}}V_0=0$.\footnote{To see this, start with $\proj_\mathcal{K}V_0=\frac{d^+}{dx}F^{**}$; from $\eqref{star}$, clearly $F^{**}=0$. Conversely, if $\proj_\mathcal{K}V_0=0$, $F^{**}$ is piecewise constant. From here, it is not difficult to argue (similar to the proof of \textit{Claim 3} in Lemma \ref{divergence-of-projection}) that $F^{**}=0$, meaning $F\geq0$ and $F(1)=0$. Here, we would like to acknowledge that our condition in this form was recognized in \cite{Hynd-Tudorascu} as a necessary condition for equilibrium in the case that $\rho_0$ is of compact support.} As observed in \cite{Hynd-Tudorascu}, since the \hyperlink{PositiveHomogeneity}{Positive Homogeneity} and \hyperlink{Contraction}{Contraction} properties of $\proj_\mathcal{K}$ give us that 
    \begin{align*}
        \|\proj_\mathcal{K}(N_0+tV_0)&-t\proj_\mathcal{K}(V_0)\|_{L^2(0,1)}=\|\proj_\mathcal{K}(N_0+tV_0)-\proj_\mathcal{K}(tV_0)\|_{L^2(0,1)}
        \\
        &\leq\|N_0+tV_0-tV_0\|_{L^2(0,1)}=\|N_0\|_{L^2(0,1)}, 
    \end{align*}
    meaning that if \eqref{star} is satisfied, 
    \begin{align}\label{N-L2Bound}
        \|N_t\|_{L^2(0,1)}=\|\proj_\mathcal{K}(N_0+tV_0)\|_{L^2(0,1)}\leq\|N_0\|_{L^2(0,1)}. 
    \end{align}
    \begin{prop}\label{uniform-integrable}
        If \eqref{star} is satisfied and $N_\infty\in L^2(0,1)$, then $\{|N_t-N_\infty|\}_{t\geq0}$ is a uniformly integrable family of functions on $(0,1)$. 
    \end{prop}
    \begin{proof}
        Let $\varepsilon>0$ and consider a set $A\subset(0,1)$ of measure less than $\delta$, where $\delta>0$ is to be decided. For any $t\geq0$, by the Cauchy-Schwarz inequality and \eqref{N-L2Bound},
        \begin{align*}
            \int_A|N_t(x)|\, dx\leq\|N_t\|_{L^2(A)}\|1\|_{L^2(A)}\leq\|N_0\|_{L^2(0,1)}\sqrt{\delta}. 
        \end{align*}
        Take $\delta<\big(\varepsilon/(2\|N_0\|_{L^2(0,1)})\big)^2$. Since $N_\infty$ is integrable, we can also take $\delta$ small enough so that $\|N_\infty\|_{L^1(A)}<\varepsilon/2$; the claim follows. 
    \end{proof}
    By means of the Vitali Convergence Theorem on the set of finite measure $(0,1)$, the above proposition allows us to make the following observation. 
    \begin{remark}\label{pointwise-to-L1}
        If $N_t\to N_\infty\in L^2(0,1)$ pointwise a.e., then $\|N_t-N_\infty\|_{L^1(0,1)}\to0$ as $t\to\infty$. Further, as the pointwise a.e. limit of nondecreasing functions, $N_\infty$ has a nondecreasing, right-continuous representative in $L^2(0,1)$. 
    \end{remark}
    This allows us to argue in terms of pointwise convergence in what follows. With this all in mind, it is not difficult to see the following, a weaker version of Theorems \ref{necessary-and-sufficient} and \ref{characterization-of-equilibrium}. 
    \begin{prop}\label{motivation-prop}
        Suppose that $F(1)=0$ and $F$ is strictly positive on $(0,1)$. Then, $(\rho,v)$ converges to the equilibrium $\rho_\infty=\delta_{Y(0)}$, $Y(0)\in\mathbb{R}$: a single, stationary, point mass. 
    \end{prop}
    The following argument provides intuition for why $F$ contains all of the information we need about equilibrium; in fact, our general plan is not so different from the proof of this proposition. 
    \begin{proof}
        By the projection formula, $N_t=\frac{d^+}{dx}G_t^{**}$. Since $F(x)>0$ on $(0,1)$,  
        \begin{align*}
            G_\infty(x)\coloneq\lim_{t\to\infty}\Big[\int_0^x N_0(w) \, dw+t\int_0^x V_0(w) \, dw\Big]=\begin{cases}
                0, &x=0,
                \\
                \infty, &x\in(0,1)
                \\
                Y(0), &x=1
            \end{cases}
        \end{align*}
        for $Y(0)$ the center of mass of $\rho_0$. We have $\lim_{t\to\infty}\conv\big(G_t\big)(x)=\conv\big(G_\infty\big)(x)$ by Proposition \ref{pointwise-sup} because $t\mapsto G_t$ is increasing for fixed $x$ and $\conv\big(G_t\big)(x)\leq xY(0)$ for all $t$. We may pass the limit through the right derivative by Corollary \ref{DerivSwap}. Finally, noting that $\conv(G_\infty)$ is a line of slope $Y(0)$ completes the proof with Remark \ref{pointwise-to-L1} in mind.
    \end{proof}
    The claim of this paper is that solutions which go to equilibrium have $F\geq0$ on $(0,1)$, but this proposition demonstrates that if $F>0$ on $(0,1)$, the equilibrium is a single point mass; however, as discussed in Section \ref{Overview}, equilibrium is nuanced, as any distribution $\rho_\infty\in\mathcal{P}_2(\mathbb{R})$ is a potential equilibrium solution! This yields the following intuition: \textit{all nuance of the equilibrium $\rho_\infty$ is encoded in the zeros of $F$}. This motivates our approach for the proof of Theorem \ref{characterization-of-equilibrium}.
\subsection{Some Lemmas}
For all of the following results, take $(\rho_0,v_0)$, $(N_0,V_0)$, $F$, and $G_t$ as in the previous section. 
\begin{lemma}\label{divergence-of-projection}
    Suppose $F(x)<0$ for some $x\in(0,1)$. Then, there exists $\xi\in(0,1)$ such that 
    \begin{align*}
        \lim_{t\to\infty}\proj_{\mathcal{K}}\Big(N_0+tV_0\Big)(\xi)=-\infty.
    \end{align*}
\end{lemma}
\begin{proof}
    We structure our argument in four claims.  
    
    \textit{Claim 1:} We can consider a ``first" absolute minimum, $m$, of $F$, occurring at $x_0\in(0,1)$. 
    
    $F$ is continuous on $[0,1]$, and hence has an absolute minimum. Since $F^{-1}(\{m\})$ is closed and bounded for $m$ being the minimum value, we may let $x_0=\inf F^{-1}(\{m\})=\min F^{-1}(\{m\})$. 

    \textit{Claim 2:} $F^{**}(x_0)=F(x_0)$.

    $F^{**}(x)=\sup\{a+bx \, | \, a,b\in\mathbb{R}\text{, } a+bu\leq F(u) \, \forall \, u\in[0,1]\}$. Consider $a=F(x_0)$, $b=0$. Clearly $F(x_0)\leq F(u)$ for all $u\in[0,1]$ since $F(x_0)$ is the minimum, so $F(x_0)\leq F^{**}(x_0)\leq F(x_0)$.

    \textit{Claim 3:} There exists a point $\xi\in(0,x_0)$ such that $F^{**}(\xi)>F(x_0)$.

    We have that $F$ is absolutely continuous with $F(0)=F(1)=0$. Further, $F^{**}(0)=0$, since 
    $m\chi_{(0,1)}\leq F$\footnote{$\chi_E$ denotes the characteristic function of the set $E$.} is convex on the closed interval $[0,1]$, so
    $0=m\chi_{(0,1)}(0)\leq F^{**}(0)\leq F(0)=0$. (By the same argument, $F^{**}(1)=0$.) Therefore, there exists $\delta>0$ small enough so that $F(x)>F(x_0)/2=m/2$ on $[0,\delta]$. Thus, any decreasing line through the point $(\delta,m)$ of sufficiently small slope\footnote{For example, take the line passing through $(0,F(x_0)/2)$.} bounds $F$ below. $F^{**}$ is bounded below by this line since it is of course convex. We may therefore pick a point $\xi\in(0,\delta)\subset(0,x_0)$ with $F^{**}(\xi)>F(x_0)$. 

    \begin{figure}
    \centering
    \begin{tikzpicture}[scale=3]
        \draw (-0.1,0)--(2.5,0);
        \draw (0,-0.918)--(0,0.385);
        \draw (2.7,0) node {$\ldots$};

        \draw (0.147,-0.05)--(0.147,0.05);
        \draw (0.147,0.125) node {$\delta$};

        \draw (1.392,-0.05)--(1.392,0.05);
        \draw (1.392,0.125) node {$x_0$};

        \draw (-0.05,-0.612)--(0.05,-0.612);
        \draw (-0.147,-0.612) node {$m$};

        \draw[fill] (0.147,-0.612) circle [radius=0.5pt];

        \draw[fill] (0.07,-0.1774) circle [radius=0.5pt];
        \draw (0.33,-0.165) node {$\scriptstyle(\xi,F^{**}(\xi))$};

        \draw[dashed] (0,-0.306)--(0.147*2,-0.306-0.306*2);

        \draw plot[smooth, variable =\x, domain=0:1, samples=100] ({\x},{\x*((\x-1)^2)*(\x-3)});
        \draw plot[smooth, variable =\x, domain=1:1.505, samples=100] ({\x},{(\x^2)*((-1+\x^2)^2)*(-3+\x^2)*pow(2.71828,-1.7*(-1.3+\x^2))});
        \draw plot[smooth, variable =\x, domain=1.5:2.5, samples=10] ({\x},{(\x^2)*((-1+\x^2)^2)*(-3+\x^2)*pow(2.71828,-1.7*(-1.3+\x^2))});

        \draw[thick] plot[smooth, variable =\x, domain=0:0.28, samples=100] ({\x},{\x*((\x-1)^2)*(\x-3)});
        \draw[thick] (0.27574,-0.394)--(1.3808,-0.61097);
        \draw[thick] plot[smooth, variable =\x, domain=1.3808:1.41, samples=100] ({\x},{(\x^2)*((-1+\x^2)^2)*(-3+\x^2)*pow(2.71828,-1.7*(-1.3+\x^2))});
        \draw[thick] (1.41,-0.60975)--(2.5,-0.30782);

        \draw (2.5,0.23) node {$F$};
        \draw (2.5,-0.405) node {$F^{**}$};
    \end{tikzpicture}
    \caption{An illustration of the argument in \textit{Claim 3}. The dashed line denotes the line constructed through the point $(\delta,m)$.}
    \label{figureLemma}
\end{figure}
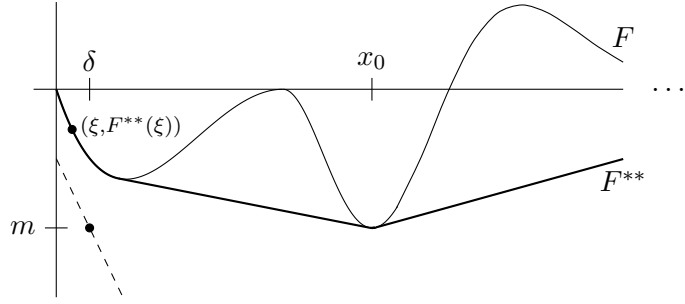   

    \textit{Claim 4:} $\displaystyle\lim_{t\to\infty}\proj_{\mathcal{K}}\Big(N_0+tV_0\Big)(\xi)=\lim_{t\to\infty}\frac{d^+ G^{**}_t}{dx}(\xi)=-\infty$.
    
    

    By proposition \ref{convex-properties}, $G_t^{**}$ is absolutely continuous on $[\xi,x_0]$ and right-differentiable everywhere, with nondecreasing right derivative. Thus, we may write
    \begin{align}
        G_t^{**}(x_0)-G_t^{**}(\xi)&
        \label{G*bound1}
        \geq
        (x_0-\xi)\frac{d^+ G_t^{**}}{dx}(\xi).
    \end{align}
    Now, we also have the inequality  
    \begin{align*}
        G_t^{**}(x_0)\leq G_t(x_0)=\int_0^{x_0} N_0(x) \, dx+tF(x_0).
    \end{align*}
    Note also that $x\mapsto\int_0^x N_0(w)dw$ is convex on $[0,1]$ since $N_0(x)$ is nondecreasing. Moreover, $F^{**}\leq F$ and is convex; thus, we have that $x\mapsto\int_0^x N_0(w)dw+tF^{**}(x)$ is a convex function bounding $G_t(x)$ from below. 
    Therefore, 
    \begin{align*}
        G_t^{**}(\xi)\geq\int_0^\xi N_0(x) \, dx+tF^{**}(\xi),
    \end{align*}
    so 
    \begin{align*}
        G_t^{**}(x_0)-G_t^{**}(\xi)\leq G_t(x_0)-\bigg(\int_0^\xi \!\! N_0(x) \, dx+tF^{**}(\xi)\bigg)=\int_\xi^{x_0} \!\! N_0(x) \, dx+t(F(x_0)-F^{**}(\xi)).
    \end{align*}
    Considering the bound \eqref{G*bound1}, we get that 
    \begin{align*}
        \frac{d^+ G_t^{**}}{dx}(\xi)\leq\frac{1}{x_0-\xi}\Big(\int_\xi^{x_0} N_0(x) \, dx+t(F(x_0)-F^{**}(\xi))\Big).
    \end{align*}
    But, by \textit{Claim 3}, $F(x_0)-F^{**}(\xi)<0$, and $\int_\xi^{x_0} N_0(x)dx$ is finite, so we are done: 
    \begin{align*}
        \lim_{t\to\infty}\Big[\int_\xi^{x_0} N_0(x) \, dx+t(F(x_0)-F^{**}(\xi))\Big]&=-\infty. \qedhere
    \end{align*}
\end{proof}

\begin{cor}\label{necessary-cor}
    If $F(x)<0$ for some $x\in (0,1)$, then the SPS to \eqref{PE} with initial conditions given by $(N_0,V_0)$ does not go to equilibrium.
\end{cor}

\begin{proof}
    Take $\xi$ as in the previous lemma. Since the right derivative is nondecreasing, (Proposition \ref{convex-properties}, \textup{(iv)}) $N_t=\proj_{\mathcal{K}}(N_0+tV_0)\to-\infty$ as $t\to\infty$ on $[0,\xi]$. 
    But since this limit is unbounded on a set of positive measure, clearly it is not integrable (or in fact $p$-integrable for any $p$). Since $N_t$ converges pointwise to something not in $L^1(0,1)$, its push-forward cannot converge to anything in the Wasserstein sense, so the solution does not converge to an equilibrium.
\end{proof}

\begin{lemma}\label{exist-finite-lemma}
    \eqref{star} is equivalent to $\displaystyle \lim_{t\to\infty} G_t^{**}(x)$ existing and being a finite convex function. 
\end{lemma}
\begin{proof}
    Suppose $F\geq0$ with $F(0)=F(1)=0$. Then, $t\mapsto G_t^{**}$ is nondecreasing (though equal to $G_0^{**}$ at $0$ and $1$) and $G_t^{**}$ is bounded above by the line from $(0,G_0(0))$ to $(1,G_0(1))$ for any $t$. Further, since $G_0$ is absolutely continuous on $[0,1]$, it has an absolute minimum. Thus, $t\mapsto G_t^{**}$ is monotone increasing and uniformly bounded below and above, so it converges to a finite, convex (by Proposition \ref{pointwise-sup}) function.  

    If instead \eqref{star} is violated, we have $G_t^{**}\leq G_t\to-\infty$ as $t\to\infty$ in the set on which $F<0$. 
\end{proof}

As a consequence of the prior lemma, Corollary \ref{DerivSwap} allows us to exchange the following limit, noting that the projection formula gives $\lim_{t\to\infty} N_t=\lim_{t\to\infty} \frac{d^+}{dx}G_t^{**}$.

\begin{cor}\label{sufficient-cor}
    If \eqref{star} is satisfied, then $\displaystyle\lim_{t\to\infty} N_t(x)$ exists and is equal a.e. to $\displaystyle\frac{d^+}{dx}\lim_{t\to\infty} G_t^{**}(x)$.
\end{cor}

The final lemma is itslef an interesting fact about solutions satisfying \eqref{star}. 
\begin{lemma}\label{N-contained}
    {Suppose \eqref{star} is satisfied and there exists $[a,b)\subset[0,1]$ such that $F(a)=F(b)=0$. Then, $N_t(x)\in [N_0(a),N_0(b)]$ for every $x\in[a,b)$ and every $t\geq0$}.
\end{lemma}
\begin{proof}
Fix $t\geq0$, $x\in[a,b)$. We will first prove that $N_t(x)\geq N_0(a)$.\footnote{If $a=0$, it is possible that $N_0(a)\coloneq\lim_{x\to a^+}N_0(x)=-\infty$; in this case, $N_t(x)\geq N_0(a)$ is trivial, so we suppose here that $N_0(a)$ is finite. Because $N_0$ is nondecreasing, we cannot have $N_0(a)=\infty$. In the second part of the proof, the trivial case $\lim_{x\to b^-}N_0(x)=\infty$ may similarly be discarded.} Since $N_t$ is nondecreasing, it suffices to prove $\proj_\mathcal{K}\big(N_0+tV_0\big)(a)\geq N_0(a)$ by the projection formula. By \eqref{star}, $t\mapsto G_t$ is nondecreasing. Moreover, since $N_0(x)$ is nondecreasing, its primitive is convex. Hence, 
\begin{align}\label{G-conv-bound}
    G^{**}_t(x)\geq\int_0^x N_0(w) \, dw \quad \text{for all } x\in[0,1].
\end{align}
Note that since $F(a)=F(b)=0$, $G_t(a)=\int_0^a N_0(x)dx$ and $G_t(b)=\int_0^b N_0(x)dx$.
However, since $G_t^{**}(x)\leq G_t(x)$, and by \eqref{G-conv-bound},
\begin{align}\label{G-conv-endpoints}
    G_t^{**}(a)=\int_0^a N_0(x) \, dx, \quad  G_t^{**}(b)=\int_0^b N_0(x) \, dx, \quad \text{and} \quad G_t^{**}(a+h)\geq \int_0^{a+h} N_0(x) \, dx
\end{align}
for $h>0$, so by the right continuity of $N_0$,
\begin{align*}
    \proj_\mathcal{K}\big(N_0+tV_0\big)(a)=\lim_{h\to 0^+}\frac{G_t^{**}(a+h)-G_t^{**}(a)}{h}\geq\lim_{h\to 0^+}\dashint_a^{a+h} N_0(x) \, dx=N_0(a).
\end{align*}

Now, we will show that $N_t(x)\leq N_0(b)$. As before, we will compute $N_t(x)$ by the projection formula. By convexity of $G_t^{**}$ with \eqref{G-conv-endpoints}, \eqref{G-conv-bound}, and the fact that $N_0$ is nondecreasing,
\begin{align*}
    N_t(x)&=\proj_{\mathcal{K}}\big(N_0+tV_0\big)(x)=\lim_{h\to0^+}\frac{G_t^{**}(x+h)-G_t^{**}(x)}{h} 
    \\
    &\leq \lim_{h\to0^+}\frac{1}{h}\bigg(\Big[\frac{(x+h)-b}{b-a}\int_a^b N_0(w) \, dw+\int_0^b N_0(w) \, dw\Big]-\int_0^x N_0(w) \, dw\bigg)
    \\
    &=\lim_{h\to0^+}\frac{1}{h}\Big((x+h-b)\dashint_a^b N_0(w) \, dw+\int_x^b N_0(w) \,dw\Big)
    \\
    &\leq \lim_{h\to0^+}\frac{1}{h}\Big((x+h-b)N_0(b)+(b-x)N_0(b)\Big)=\lim_{h\to0^+}\frac{hN_0(b)}{h}=N_0(b). \qedhere
\end{align*}

\end{proof}

\subsection{Main Results}
    We are now ready to state and prove our main results. Remark \ref{pointwise-to-L1} and Corollaries \ref{necessary-cor} and \ref{sufficient-cor} have established that \eqref{star} is in fact the condition we desire: 
    \begin{thm}[Necessary and Sufficient Condition]\label{necessary-and-sufficient}
        The SPS to \eqref{PE}  with initial conditions given by $(N_0, V_0)$ goes to equilibrium as $t\to\infty$ if and only if
        \begin{align*}\label{star2}\tag{$\bigstar$}
            \int_0^1 V_0(w) \, dw=0 \quad\text{and}\quad \int_0^x V_0(w) \, dw\geq0 \quad\text{for all}\quad x\in(0,1). 
        \end{align*}        
    \end{thm}
    Further, for a solution given by $(N_t,V_t)$, the solution originating at $(N_{t_0},V_{t_0})$, $t_0>0$, will be identical to $(N_t,V_t)$, but shifted by $t_0$ in time. (In the literature, this is referred to as the semigroup property.) Thus, the bound \eqref{N-L2Bound} applies for any $t_0$, not just $0$, since if $(N_t,V_t)$ goes to equilibrium, so does the solution originating at $(N_{t_0},V_{t_0})$:
    \begin{cor}
        If the solution to \eqref{PE} given by $(N_0,V_0)$ goes to equilibrium, then $t\mapsto\|N_t\|_{L^2(0,1)}$ is nonincreasing; that is, the second moment of $\rho_t$ is decreasing in time. 
    \end{cor}
    In \cite{Hynd-Tudorascu}, it is also observed that \eqref{star} follows in the case that $\{N_t\}_{t\geq0}$ is uniformly bounded in $L^1(0,1)$. Thus, solutions which do not converge to equilibrium spread out arbitrarily far. 

\begin{thm}[Characterization of Equilibrium]\label{characterization-of-equilibrium}
    Suppose $(N_t,V_t)$ captures the SPS to \eqref{PE} originating at $(N_0,V_0)$ satisfying \eqref{star}, and take $F$ as in \eqref{F-def}. Then, we have the following:
    \begin{enumerate}
        \item On intervals $[a,b)\subset[0,1]$ such that $F(a)=F(b)=0$ and $F>0$ on $(a,b)$, 
        \begin{align*}
            N_\infty(x)=\dashint_a^b N_0(w) \, dw. 
        \end{align*}
        \item On intervals $[c,d)\subset[0,1]$ with $F$ identically $0$, $N_t(x)=N_0(x)$ for all $t$, so $N_\infty(x)=N_0(x)$.
    \end{enumerate}
     
\end{thm}

The first part of the above theorem should be thought of as giving the position where a particle (a Dirac measure) of mass $b-a$ forms (at the center of mass of that ``chunk of mass") as the solution approaches equilibrium. Note that although $a, b\in [0,1]$, this theorem applies to mass with unbounded support (i.e., it is possible that $N_0(0^+)=-\infty$ or $N_0(1^-)=+\infty$). The second part of the above theorem tells us which initial stationary mass stays stationary for all time. Since  $F$ is continuous the interval $[0,1)$ may always be written as a union of intervals for which one of the two above cases apply. Notably, then, the only possibilities for equilibrium are that mass in any given interval collapses to a Dirac measure, or it remains stationary for all time. 

\begin{proof}[Proof of Theorem \ref{characterization-of-equilibrium}]
Suppose \eqref{star} is satisfied and take $G_t$ as in \eqref{G-def}. By Corollary \ref{sufficient-cor}, 
$$N_\infty(x)=\frac{d^+}{dx}\lim_{t\to\infty}G_t^{**}(x).$$ 
Consider first an interval $[a,b)$ as described in the first part of the theorem. We may observe (similarly to the proof of Proposition \ref{motivation-prop}) that on the interval $[a,b]$, 
\begin{align*}
    G_\infty(x)\coloneq\lim_{t\to\infty}\Big[\int_0^x N_0(w) \, dw+tF(x)\Big]=\begin{cases}
        \displaystyle\int_0^a N_0(w) \, dw, &x=a
        \\
        \infty, &x\in(a,b)
        \\
        \displaystyle\int_0^b N_0(w) \, dw, &x=b, 
    \end{cases}
\end{align*} 
so the greatest convex function on $[a,b]$ bounding $G_\infty$ below (which by Proposition \ref{pointwise-sup} is equal to the limit of such bounding functions as $t\to\infty$\footnote{The limit is finite because $\conv (G_t)$ is bounded above by a line on $[0,1]$ for all $t$.}) is a line with slope (right derivative)
\begin{align*}
    \frac{1}{b-a}\Big(\int_0^b N_0(w) \, dw-\int_0^a N_0(w) \, dw\Big)=\dashint_a^b N_0(w) \, dw. 
\end{align*}
Call this line $\gamma_{[a,b]}:[a,b]\to\mathbb{R}$. Of course, $\conv(G_\infty)$ on $[0,1]$ may not be greater than $\gamma_{[a,b]}$ on $[a, b]$, so we need to show that there exists a convex function coinciding with $\gamma_{[a,b]}$ on $[a, b]$ which bounds $G_\infty$ below on $[0,1]$. $G_\infty\coloneq\lim_{t\to\infty}G_t=\infty$ anywhere where $F$ is nonzero; thus, in taking $G_\infty^{**}\coloneq\conv(G_\infty)$, \textit{we need only consider the zeros of} $F$. Let $\mathcal{Z}=\{x\in[0,1] \ | \ F(x)=0\}$. Take $\gamma:[0,1]\to\mathbb{R}$ to be equal to the primitive of $N_0$ on $\mathcal{Z}$; then, take it to be linear on each connected component of $[0,1]\backslash\mathcal{Z}$ so that it is continuous on $[0,1]$. (See figure \ref{gamma-construction}.) Since the primitive of $N_0$ is convex, as $N_0$ is increasing, one can observe that $\gamma$ is convex since its right derivative is nondecreasing. 

\begin{figure}
    \centering
    \begin{tikzpicture}[scale=3]
        \draw (-0.05*2,0)--(1*2,0);
        \draw (0,-0.35)--(0,1);

        \draw (0.333*2,-0.025)--(0.333*2,0.025);
        \draw (0.5*2,-0.025)--(0.5*2,0.025);

        \draw (0.333*2,0.08) node {$a$};
        \draw (0.5*2,0.08) node {$b$};

        \draw (0,1.3) node {$\vdots$};
        \draw (0,1.5) node {$\infty$};

        \draw (0.4*2,-0.17) node {$\scriptstyle{\gamma_{[a,b]}}$};
        \draw (0.8*2,0.5) node {$\gamma$};
        \draw (1.15*2,0.4) node {$x\mapsto\int_0^x N_0(m)dm$};
        \draw (0.7*2,1.6) node {$G_\infty$};

        \draw[fill] (0,0) circle [radius=0.5pt];
        \draw[fill] (0.2*2,-0.2173) circle [radius=0.5pt];
        \draw[fill] (0.333*2,-0.265) circle [radius=0.5pt];
        \draw[fill] (0.5*2,-0.20833) circle [radius=0.5pt];
        \draw[fill] (0.6*2,-0.108) circle [radius=0.5pt];
        \draw[fill] (0.75*2,0.14063) circle [radius=0.5pt];
        \draw[fill] (1*2,0.83333) circle [radius=0.5pt];

        \draw[thick] plot[smooth, variable =\x, domain=0:0.2*2] ({\x},{2*(\x/2)^2-(3/2)*(\x/2)+(1/3)*(\x/2)^3});
        \draw[thick] plot[smooth, variable =\x, domain=0.5*2:0.6*2] ({\x},{2*(\x/2)^2-(3/2)*(\x/2)+(1/3)*(\x/2)^3});
        
        \draw[thick] (0.2*2,1.5)--(0.5*2,1.5);
        \draw[thick] (0.6*2,1.5)--(1*2,1.5);

        \filldraw[fill=white, draw=black] (0.2*2,1.5) circle [radius=0.5pt];
        \filldraw[fill=white, draw=black] (0.3333*2,1.5) circle [radius=0.5pt];
        \filldraw[fill=white, draw=black] (0.5*2,1.5) circle [radius=0.5pt];
        \filldraw[fill=white, draw=black] (0.6*2,1.5) circle [radius=0.5pt];
        \filldraw[fill=white, draw=black] (0.75*2,1.5) circle [radius=0.5pt];
        \filldraw[fill=white, draw=black] (1*2,1.5) circle [radius=0.5pt];

        \draw[thin] (0.2*2,-0.2173)--(0.333*2,-0.265);
        \draw[thin] (0.333*2,-0.265)--(0.5*2,-0.20833);
        \draw[thin] (0.6*2,-0.108)--(0.75*2,0.14063);
        \draw[thin] (0.75*2,0.14063)--(1*2,0.83333);

        \draw[dashed] plot[smooth, variable =\x, domain=0:1*2] ({\x},{2*(\x/2)^2-(3/2)*(\x/2)+(1/3)*(\x/2)^3});
    \end{tikzpicture}
    \caption{Construction of the convex function $\gamma$.}
    \label{gamma-construction}
\end{figure}
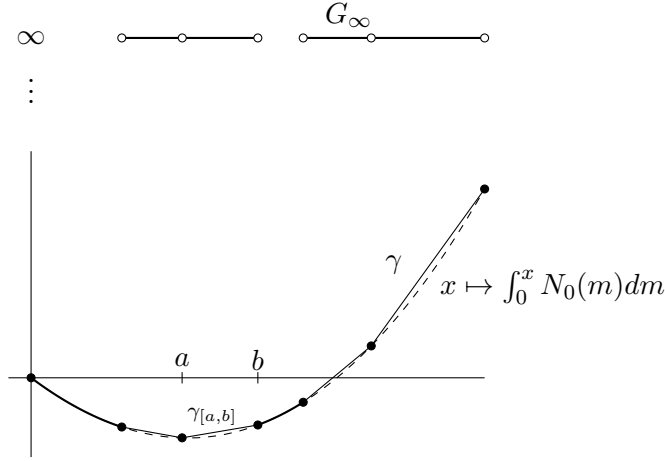   

To show the second part, consider $[c,d)$ as described in the theorem, and let $x\in[c, d)$, $t>0$. By Lemma \ref{N-contained}, since $F$ is $0$ on $[x,x+\delta)$ for $\delta>0$ small enough, $N_t(x)\in[N_0(x), N_0(x+\delta)]$, from which the result follows by the right continuity of $N_0$. 
\end{proof}

\begin{thm}[Time of Collapse to Equilibrium]
    Suppose $(N_t,V_t)$ captures the SPS to \eqref{PE} originating at $(N_0,V_0)$, and discard the trivial case where $V_0=0$ a.e. Take $F$ as in \eqref{F-def} and let 
    \begin{align*}
        \mathcal I\coloneq\{[a,b]\subseteq [0,1] \ \ | \ \ F(a)=F(b)=0, \ F(x)>0\ \forall x\in(a, b)\}.
    \end{align*}
    If the solution goes to equilibrium, the (potentially infinite) time of collapse is given by 
    \begin{align}\label{time-to-collapse}
        T_c=\sup_{[a,b]\in\mathcal I}~\sup_{x\in (a,b)}\frac{\displaystyle\int_x^b N_0(w) \, dw \, + \, \frac{x-b}{b-a}\int_a^b N_0(w) \, dw}{\displaystyle\int_0^x V_0(w) \, dw}. 
    \end{align}
\end{thm}
\begin{proof}
    $F(0)=F(1)=0$ and $F$ is continuous, so $[0,1]$ can be written as a union of intervals of the type $[a,b]$ where $F(a)=F(b)=0$, $F>0$ on $(a,b)$, and intervals on which $F$ is identically $0$. On intervals of the second type, the previous theorem tells us that $N_t=N_\infty$ for all $t$, so they do not contribute to the time of collapse. Thus, we need only consider intervals of the first type. Suppose $[a,b]$ is one such interval. By the previous theorem and the projection formula,  
    \begin{align*}
        \frac{d^+}{dx}G_t^{**}(x)=N_t(x) \ \to \ N_\infty(x)=\dashint_a^b N_0(w) \, dw \quad \text{on }[a,b]\text{ as }t\to\infty. 
    \end{align*}
    In other words, noting that $G_t(a)=\int_0^a N_0(w)dw$ and $G_t(b)=\int_0^b N_0(w)dw$ for all $t\geq0$, it is the only possibility that $N_t=N_\infty$ on this interval if and only if $t$ satisfies
    \begin{align*}
        \int_0^x N_0(w) \, dw+t\int_0^x V_0(w) \, dw=G_t(x)\geq\int_a^b N_0(w) \, dw\frac{x-b}{b-a}+\int_0^b N_0(w) \, dw 
    \end{align*}
    for all $x\in(a,b)$. For finite-time collapse to occur, that needs to be true for every $x$ in every such interval. Solving for $t$, we see that the time of collapse is given by \eqref{time-to-collapse}. \qedhere 
    



\end{proof}

\subsection{The case $p>1$}\label{p>1}
    As stated in Definition \ref{equilibrium-def}, our results apply to convergence in the sense of the $1$-Wasserstein distance. The arguments in this paper imply that if \eqref{star} is not satisfied, $\rho_t$ does not converge to equilibrium in the Wasserstein sense for any $p$. 

    Relevant here are the results in \cite{Hynd-Tudorascu}, which concern solutions confined to a compact set for all time. Suppose $\rho_t$ is supported in $[a,b]$ regardless of $t\geq0$. If $W_1(\rho_t,\rho_\infty)\to0$ as $t\to\infty$, we can observe that for all $t$, $a\leq\inf N_t\leq\sup N_t\leq b$, so the bound 
    \begin{align*}
        W_p(\rho_t,\rho_\infty)&=\Big(\int_0^1|N_t-N_\infty|^p\,dx\Big)^{1/p}\leq\Big(\int_0^1|N_t-N_\infty||b-a|^{p-1}\,dx\Big)^{1/p}
        \\
        &=|b-a|^{\frac{p-1}{p}}\big(W_1(\rho_t,\rho_\infty)\big)^{1/p}
    \end{align*}
    gives us that $W_p(\rho_t,\rho_\infty)\to0$ for any $p\in(0,\infty)$. If $\spt\rho_0\subseteq[a,b]$ and \eqref{star} is satisfied, observe that Lemma \ref{N-contained} applied to the entire interval tells us that $\spt\rho_t\subseteq[a,b]$ for all $t>0$. 

    Our results can be extended to $p\in(1,2)$ by an application of Hölder's inequality and \eqref{N-L2Bound}, if one desires. Of course, it would be particularly desirable to extend the results to $p=2$. We are not currently aware of an example for which $W_1$ convergence occurs but $W_2$ convergence does not. However, it is possible that some solutions, while maintaining finite second moment and finite total momentum, remain ``spread out" enough that $N_t$ behaves pathologically near $0$ or $1$. 
    
    If a solution captured by $(N_t,V_t)$ satisfies \eqref{star}, convergence in the $2$-Wasserstein sense can be obtained in the case that there exists a $\delta>0$ so that $t\mapsto N_t(x)$ is monotone increasing for $x\in(0,\delta]$ and monotone decreasing for $x\in[1-\delta,1)$. Physically, this means that, for all time, there is no outward velocity outside of some fixed compact set. In this case, since $N_t$ is nondecreasing for any $t$, we obtain the following bounds on $N_t(x)$. 
    \begin{align*}
        N_0(\delta)\leq N_t(\delta)\leq N_t(x)\leq N_t(1-\delta)\leq N_0(1-\delta), \quad &x\in[\delta,1-\delta],
        \\
        N_0(x)\leq N_t(x)\leq N_t(1-\delta)\leq N_0(1-\delta), \quad &x\in(0,\delta],
        \\
        N_0(\delta)\leq N_t(\delta)\leq N_t(x)\leq N_0(x), \quad &x\in[1-\delta,1).
    \end{align*}
    As a consequence, $|N_t(x)|\leq H(x)\coloneq\max\{|N_0(x)|,|N_0(\delta)|,|N_0(1-\delta)|\}$ on $(0,1)$. Since $N_0$ is nondecreasing and square-integrable, $H\in L^2(0,1)$. Thus, we have $\|N_t-N_\infty\|_{L^2(0,1)}$ as $t\to\infty$ by the $L^2$ version of the Dominated Convergence Theorem. Of course, more meaningful would be a result depending only on the initial condition $(N_0,V_0)$. This is difficult because convergence in $L^2(0,1)$ depends on the behavior of $N_t$ near $0$ and $1$, and yet, any analysis of the projection formula must consider global properties of $N_0+tV_0$ when taking the convex envelope. Finally, we note that this condition is not necessary for convergence in $W_2$. We conclude with an example of a solution not satisfying the above condition, but which converges in $W_p$ for all $1\leq p\leq\infty$. 
    \begin{example}
        Consider the initial conditions given by 
        \begin{align}\label{example2-initial-condition}
            \begin{cases}
                N_0=\displaystyle\sum_{n=0}^\infty\big[2n\chi_{[1-2^{-n},1-3\cdot2^{-2-n})}+(2n+1)\chi_{[1-3\cdot2^{-2-n},1-2^{-1-n})}\big],
                \\
                V_0=\displaystyle\sum_{n=0}^\infty\big[\frac{1}{2}\chi_{[1-2^{-n},1-3\cdot2^{-2-n})}-\frac{1}{2}\chi_{[1-3\cdot2^{-2-n},1-2^{-1-n})}\big]. 
            \end{cases}
        \end{align}
        One may verify that $\|N_0\|_{L^2(0,1)}=\sum_{n=0}^\infty 2^{-n}\big(2n^2+n+1/4\big)<\infty$ and $\quad \|V_0\|_{L^2(0,1)}=1/2$, and that the solution originating at \eqref{example2-initial-condition} to \eqref{PE} is captured by 
        \begin{align*}
                N_t=\sum_{n=0}^\infty\big[(2n+\frac{t}{2})\chi_{[1-2^{-n},1-3\cdot2^{-2-n})}+(2n+1-\frac{t}{2})\chi_{[1-3\cdot2^{-2-n},1-2^{-1-n})}\big] \quad (0<t<1),
        \end{align*}
        $V_t=V_0$ $(0<t<1)$, and that, for $t>1$, 
        \begin{align*}
            N_\infty=N_t=N_1=\sum_{n=0}^\infty(2n+\frac{1}{2})\chi_{[1-2^{-n},1-2^{-1-n})}, \quad V_t=0.
        \end{align*}
        Thus, for $1\leq p\leq\infty$, $\rho_t={N_t}_\#\mathcal{L}$, $\rho_\infty={N_\infty}_\#\mathcal{L}$, and $0\leq t\leq1$, we compute 
        \begin{align*}
            W_p(\rho_t,\rho_\infty)&=\|N_t-N_\infty\|_{L^p(0,1)}
            =\|\frac{t}{2}-\frac{1}{2}\|_{L^p(0,1)}=\frac{1}{2}-\frac{t}{2}. 
        \end{align*}
    \end{example}

\section{Acknowledgments}
    We would like to emphatically and sincerely thank Professor Adrian Tudorascu, our mentor in a related project on the repulsive pressureless Euler-Poisson system, and Jack Curtis, who worked with us invaluably on the same project. The ideas present in this paper were originally intended for the repulsive scheme, and after they were unsuccessful there we realized they could apply to this case. We would also like to thank Professor Tsikkou and the REU Site in Undergraduate Research in Applied Analysis at West Virginia University, where these ideas were initially generated. We would like to thank the National Science Foundation (DMS-2349040), and the Department of Mathematics and Statistics for their support and hosting of the REU. 


\end{document}